\documentclass{amsart}

\usepackage{mathabx}
\usepackage{extpfeil}
\usepackage{hyperref}
\usepackage{amssymb, fancyhdr, stmaryrd, amsthm, mathrsfs, textcomp}

\usepackage[utf8]{inputenc}
\usepackage{amsfonts,amsmath,amscd,amssymb}
\usepackage{dsfont,mathtools}
\usepackage{xspace}
\usepackage{fixmath}
\usepackage{tikz-cd}
\usepackage{caption}
\usepackage{subcaption}
\usepackage{graphicx}
\usepackage{listings}
\usepackage{faktor}
\usepackage{calligra}

\def\R{\ensuremath\mathbb{R}}
\def\C{\ensuremath\mathbb{C}}

\def\bI{\ensuremath\mathbb{I}}
\def\bJ{\ensuremath\mathbb{J}}
\def\bM{\ensuremath\mathbb{M}}
\def\bL{\ensuremath\mathbb{L}}
\def\N{\ensuremath\mathbb{N}}
\def\Z{\ensuremath\mathbb{Z}}
\def\F{\ensuremath\mathbb{F}}

\def\a{\ensuremath\alpha}

\def\G{\ensuremath\Gamma}

\def\to{\ensuremath\rightarrow}

\def\<{\ensuremath\langle}
\def\>{\ensuremath\rangle}

\DeclareMathAlphabet{\mathcalligra}{T1}{calligra}{m}{n}

\DeclareMathOperator{\SL}{SL}

\DeclareMathOperator{\Hom}{Hom}

\DeclareMathOperator{\ad}{ad}

\DeclareMathOperator{\voll}{Vol}
\DeclareMathOperator{\cl}{cl}

\newtheorem{thm}{Theorem}[section]
\newtheorem{thmM}{Main Theorem}
\newtheorem{lemma}[thm]{Lemma}
\newtheorem{prop}[thm]{Proposition}
\newtheorem{cor}[thm]{Corollary}
\newtheorem{con}[thm]{Conjecture}
\newtheorem{que}[thm]{Question}

\theoremstyle{definition}
\newtheorem{defn}[thm]{Definition}

\newcommand{\vc}{\ensuremath{\mathbf{c}}\xspace}
\newcommand{\ve}{\ensuremath{\mathbf{e}}\xspace}
\newcommand{\vf}{\ensuremath{\mathbf{f}}\xspace}

\newcommand{\vs}{\ensuremath{\mathbf{s}}\xspace}
\newcommand{\vt}{\ensuremath{\mathbf{t}}\xspace}

\newcommand{\vx}{\ensuremath{\mathbf{x}}\xspace}
\newcommand{\vy}{\ensuremath{\mathbf{y}}\xspace}

\newcommand{\sD}{\mathcal{D}}

\newcommand{\sF}{\mathcal{F}}

\newcommand{\sU}{\mathcal{U}}

\newcommand{\sX}{\mathcal{X}}
\newcommand{\sY}{\mathcal{Y}}
\newcommand{\sZ}{\mathcal{Z}}

\newcommand{\mg}{\mathfrak{g}}

\newcommand{\mh}{\mathfrak{h}}

\begin{document}
\title[Lattices in Kac-Moody Groups]{Cocompact Lattices in Locally Pro-$p$-complete Rank 2 Kac-Moody Groups}
\author{Inna Capdeboscq}
\email{I.Capdeboscq@warwick.ac.uk}
\address{Department of Mathematics, University of Warwick, Coventry, CV4 7AL, UK}
\author{Katerina Hristova}
\email{K.Hristova@uea.ac.uk}
\address{School of Mathematics, University of East Anglia, Norwich, NR4 7TJ, UK}
\author{Dmitriy Rumynin}
\email{D.Rumynin@warwick.ac.uk}
\address{Department of Mathematics, University of Warwick, Coventry, CV4 7AL, UK
  \newline
\hspace*{0.31cm}  Associated member of Laboratory of Algebraic Geometry, National
Research University Higher School of Economics, Russia}
\thanks{The third author was partially supported by the Russian Academic Excellence Project `5--100' and by the Max Planck Society.}

\date{February 25, 2019}

\begin{abstract}
  We initiate an investigation of lattices in a new class of locally compact groups, so called locally pro-$p$-complete Kac-Moody groups. We discover that in rank 2 their cocompact lattices are particularly well-behaved:
under mild assumptions, a cocompact lattice
in this completion contains no elements of order $p$. 
This statement is still an open question for 
the Caprace-R\'emy-Ronan completion.
Using this, modulo results of Capdeboscq and Thomas,
we classify edge-transitive cocompact lattices
and describe a cocompact lattice of minimal covolume.
\end{abstract}

\maketitle

Given a locally compact group $H$, it is an intriguing question to find a lattice
attaining minimal covolume in $H$, if, of course, it exists. 
The earliest important result along these lines is
a classical theorem of Siegel that the minimal covolume lattice in $\SL_2(\R)$ is
the triangle $(2,3,7)$-group. In particular, it is cocompact.

Over a non-Archimedean local field the nature of the minimal covolume lattice changes drastically.
As shown by Lubotzky \cite{Lu}, the lattice of minimal
covolume in 
 $\SL_2 (\F_q ((t))\, )$ is $\SL_2(\F_q [t^{-1}])$.
It is no longer cocompact. 

Another interesting question is to find
a cocompact lattice of minimal covolume, i.e.,
minimal on the set of cocompact lattices.
In the same paper \cite{Lu} Lubotzky
answers this question as well for 
$\SL_2 (\F_q ((t))\, )$. Notice that for locally compact groups over a non-Archimedean local field
it is rare for the cocompact lattices to exist.
As proven by Borel and Harder \cite{BH},
the only  simple Chevalley groups
$G(\F_q ((t))\, )$ that admit cocompact lattices
are those of type $A_{n-1}$, i.e., the groups 
$\SL_n (\F_q ((t))\, )$.

The group $\SL_2 (\F_q ((t))\, )$ is a basic example
of a locally compact Kac-Moody group of rank $2$ over the finite field $\F_q$ of $q=p^n$ elements.
Let $G=G(\F_q)$ be a minimal Kac-Moody group of rank $2$ (with a minor restriction on the Cartan matrix).
Capdeboscq and Thomas  \cite{CaTh} construct and study lattices in the topological Kac-Moody group $\overline{G}$, the Caprace-R\'emy-Ronan completion of $G$.
They find the minimal covolume lattices in $\overline{G}$ among
the class of cocompact lattices without elements of order $p$ \cite{CaTh}.
This  order $p$ restriction is motivated by the fact that the cocompact lattices in  $\SL_2 (\F_q ((t))\, )$ do not contain any elements of order $p$. It feels reasonable that the same phenomenon should hold in a general topological Kac-Moody rank-$2$ group.
Capdeboscq and Thomas even conjecture that cocompact lattices in such groups contain no elements of order $p$. 


However, there are other topological Kac-Moody groups, for instance, the Ma\-thieu-Ro\-us\-seau group  $G^{ma+}$, the Carbone-Garland group $G^{c\lambda}$ and the locally pro-$p$-complete group $\widehat{G}$, introduced by Capdeboscq and Rumynin \cite{CaRu}. Denoting the closure of $G$ in  $G^{ma+}$ by $\overline{G}^+$ (note that $G^{ma+}=\overline{G}^+$ in many but not all cases), we have surjective group homomorphisms
$$\widehat{G}\twoheadrightarrow \overline{G}^+\twoheadrightarrow G^{c\lambda}\twoheadrightarrow \overline{G}.$$
Sometimes these groups coincide, for instance, in the untwisted affine case of rank at least 3 \cite[Prop. 3.2]{CaRu}. 
On other occasions, the groups are distinct, for example, $\widehat{G} \neq \overline{G}$ in rank 2.
The precise relations among these groups are highly intriguing
(see \cite{kn::CR,CaRu,kn::E,Mar2,kn::Rou} for several discussions on this topic).
It is fair to conclude that the study of all the completions is interesting:  this yields new information about all of them as well as the original group $G$.


In the present paper we study cocompact lattices in the local pro-$p$-completion $\widehat{G}$ in rank $2$.
We prove that a cocompact lattice in $\widehat{G}$ contains no elements of order $p$. 
We establish pushing and pulling procedures connecting
cocompact lattices in $\widehat{G}$ with
cocompact lattices in $\overline{G}$ without elements of order $p$. 
As an upshot, we describe the minimal covolume cocompact lattices in $\widehat{G}$.
This demonstrates that the lattices in the local pro-$p$-completion $\widehat{G}$ is an interesting subject, deserving further attention of mathematicians. It would be equally important to investigate lattices in the other complete Kac-Moody groups.

Let us describe the content of this paper section-by-section.
In Section~\ref{one} we set up the scene defining
the key Kac-Moody groups and discussing their structure.

In Section~\ref{two} we study the elements of order $p$ 
and the cocompact lattices in $\widehat{G}$. We prove one of
the main results
of this paper
that enable our investigation of covolumes later on:
\begin{thmM} (fusion of Theorems~\ref{P2_hatG} and \ref{P1_hatG})
  Let $A$ be a $2\times 2$ generalised Cartan matrix with
  all $|a_{ij}|\geq 2$.
  Let $\sD$ be a   root datum of type $A$.
The following statements hold for 
  the corresponding (to $\sD$) locally pro-$p$-complete
  Kac-Moody group $\widehat{G}$ 
  over the field of $q=p^a$ elements:
\begin{enumerate}
\item
  A cocompact lattice $\G$ of $\widehat{G}$ contains no elements of order $p$. 
\item   Any element of order $p$ in $\widehat{G}$ is contained in a conjugate
  of the subgroup ${\sU}$ of $\widehat{G}$.
\end{enumerate}
\end{thmM}

In Section~\ref{three}
we define  the push-forward and the pull-back of cocompact lattices.
If $\G\leq\widehat{G}$ is a cocompact lattice, 
then its push-forward 
$\pi(\G)\leq\overline{G}$ is a cocompact lattice.
If $\G\leq G\leq \overline{G}$ is a cocompact lattice, 
then its pull-back
$\G\leq G\leq \widehat{G}$ is a cocompact lattice.
Notice that both push-forward and pull-back preserve
the isomorphism class of a cocompact lattice.

In Section~\ref{four}
we compare covolumes of cocompact lattices in $\widehat{G}$ and $\overline{G}$.
After a suitable normalisation the covolumes do not change
and could be computed on the set $\sX_0$ of
vertices of the Tits building of $G$
(see Proposition~\ref{Haar_comparison}):
$$
  \widehat{\mu} ({\G\setminus\widehat{G}}) =
    \overline{\mu} ({\G\setminus\overline{G}}) = 
    \sum_{[\vx] \in \G\setminus\sX_0} \frac{1}{|\G_\vx|} \; .
$$
    
In Section~\ref{five} we utilise the results of Capdeboscq
and Thomas \cite{CaTh} about cocompact lattices in $\overline{G}$.
We prove the second main result of this paper:
  \begin{thmM}
\label{main_th}
  Let $A$ be a symmetric $2\times 2$ generalised Cartan matrix with
  all $|a_{ij}|\geq 2$.
  Let $\sD$ be a simply-connected root datum of type $A$.
The following statements hold for 
  the corresponding (to $\sD$) locally pro-$p$-complete
  Kac-Moody group $\widehat{G}$ 
  over the field of $q=p^a$ elements:
\begin{enumerate}
\item
  $\widehat{G}$ 
  admits a cocompact lattice.
\item
  If $q \geq 514$, then  
  there exist
  $\delta \in \{ 1,2,4\}$ such that
  $$ \min \{ \widehat{\mu}( \G \setminus \widehat{G})
  \,\mid\, \G \mbox{ is a cocompact lattice}\}= \frac{2}{(q+1) |Z({G})| \delta} \; .$$
\end{enumerate}
\end{thmM}
  
In Section~\ref{six} we discuss cocompact lattices in $\widehat{G}$
for more general Cartan matrices.
We formulate several questions and conjectures facilitating
further research of locally pro-$p$-complete groups $\widehat{G}$ and their cocompact lattices.

\section{Two Completions of Kac-Moody Groups}
\label{one}

Let $\F=\F_q$ be a finite field, where $q=p^a$, for some prime $p$. 
Let 
$$ A= 
\begin{pmatrix}
a_{11} & a_{12} \\
a_{21} & a_{22} \\
\end{pmatrix}
=
\begin{pmatrix}
 2 &  a_{12} \\
 a_{21} & 2
\end{pmatrix}$$
be a $2\times 2$ generalised Cartan matrix
with $\max (a_{21},a_{12}) \leq -2$. 

Recall that 
{\em a root datum of type $A$} is a quintuple
$\sD = (A, \sZ, \sY, \Pi, \Pi^\vee)$ where 
\begin{itemize}
\item $\sY$ is a finitely generated free abelian group,
\item $\sZ=\Hom_{\Z}(\sY, \Z)$ is its dual group, 
\item $\Pi=\{ \alpha_1, \alpha_2 \} \subset \sZ$ is the (ordered) set of simple roots,
\item $\Pi^\vee=\{ \a_1^\vee, \alpha_2^\vee \} \subset\sY$ is the (ordered) set of simple coroots,
\end{itemize}
satisfying a single axiom $\alpha_i(\alpha_j^\vee)=a_{ji}$ for all $i,j \in \{1,2\}$.

The Weyl group of $A$ is the infinite dihedral group
$$W= \langle w_1, w_2 \ |\ w_1^2= w_2^2 \rangle \; .$$
We use the standard notation $[w_iw_j]_m \coloneqq w_i w_jw_i \ldots$
($m$ symbols $w$ in the right hand side) for elements of $W$. 
%
The set of real roots is
$$\Phi=\{ w \alpha_1, w \alpha_2 \ |\ w \in W\}.$$ 
Denote by $\Phi_+$ the set of positive real roots and $\Phi_{-}$ the set of negative real roots.
The set $\Phi_+$ can be written as a disjoint union of the two sets $\Phi_+^1$ and $\Phi_+^2$:
$$\Phi_+^1 \coloneqq \{ \alpha_1, w_1\alpha_2, w_1 w_2\alpha_1, w_1 w_2w_1\alpha_2,\ldots ,(w_1w_2)^n\alpha_1,(w_1w_2)^n w_1\alpha_2,\ldots  \}, $$
$$\Phi_+^2 \coloneqq \{ \alpha_2, w_2\alpha_1, w_2 w_1\alpha_2, w_2 w_1w_2\alpha_1,\ldots ,(w_2 w_1)^n\alpha_2,(w_2 w_1)^n w_2\alpha_1,\ldots  \}.$$
We also need the sets $-\Phi_+^1\coloneqq \{ -\alpha \,\mid\, \alpha \in \Phi_+^1\}$ and $-\Phi_+^2$.

Let $G= G_{\sD} (\F)$ be the Kac-Moody group over the field $\F$
associated to a root datum $\sD$ of type $A$
(cf. \cite{CaCh,Ti,CaKRu}).
To define it, we introduce an additive group  
$$
U_\alpha \coloneqq  \{ x_\alpha(\vt) \, \mid \, \vt \in \F\} \, , 
\ \ \ 
x_\alpha(\vt)x_\alpha(\vs)=x_\alpha(\vt+\vs) \, ,
\ \ \ 
U_\alpha \cong (\F,+)
$$
for each real root $\alpha$
and the torus
$$H = \sY \otimes_{\Z} \F^{\times} \, .$$
Let $F$ be the free product of $H$ and all $U_\alpha$, $\alpha\in \Phi$.
We require the following elements of $F$ defined for $\alpha \in \Phi$,
$\vt\in \F^\times$, $i \in \{1,2\}$:
\begin{itemize}
\item $n_\alpha (\vt) \coloneqq x_{\alpha}(\vt) x_{-\alpha}(\vt^{-1}) x_{\alpha}(\vt)$, \ \
  $n_i(\vt) \coloneqq n_{\alpha_i}(\vt)$, 
\item $h_\alpha (\vt)\coloneqq n_{\alpha}(\vt)n_{\alpha}(1)^{-1}$, \ \ 
  $h_i(\vt)\coloneqq n_{\alpha_i}(\vt)$.
\end{itemize}
The Kac-Moody group $G= G_{\sD} (\F)$ of type $A$ is
a quotient group of $F$
by the following relations:
\begin{enumerate}
\item $h_i(\vt)=\alpha^{\vee}_i \otimes \vt$,
\item $(y\otimes \vt) x_{\alpha_i}(\vs)(y\otimes \vt)^{-1}=x_i(\vt^{\alpha_i(y)}\vs)$, 
\item $n_i(1) (y \otimes \vt) n_i^{-1}(1)=w_i(y) \otimes \vt$, 
\item $n_i(1) x_{\alpha}(\vt) n_i(1)^{-1}=x_{w_i(\alpha)}(\epsilon_{i, \alpha}\vt)$
  for uniquely determined
  $\epsilon_{i, \alpha} \in \{-1,1\}$,
\item
  $x_\alpha(\vt)x_\beta(\vs)= x_{\beta}(\vs)x_{\alpha}  (\vt)$, 
  if $\{ \alpha, \beta\} \subset \Phi_+^1 \cup -\Phi_+^2$ or
  $\{\alpha, \beta\} \subset \Phi_+^2 \cup -\Phi_+^1$.
\end{enumerate}
The choice of $\epsilon_{i, \alpha}$
in the relation~(4)
depends on the events in the corresponding Kac-Moody
algebra $\mg$ over the complex numbers (cf. \cite{CaCh}).
Let us elaborate. 
The Lie algebra $\mg$ is generated by
$$\ve_1, \ve_2, \vf_1, \vf_2
\ \mbox{ and } \
\mh = \sY \otimes_{\Z} \C \, .$$
Using the adjoint representation $\ad (\vx) (\vy) = [\vx,\vy ]$, 
we define the following operators on $\mg$: 
$$
\eta_i \coloneqq \exp (\ad (\ve_i))\exp (\ad (\vf_i))\exp (\ad (\ve_i))
\, , \ \ 
[\eta_i \eta_j]_m \coloneqq \eta_i \eta_j \eta_i \ldots
\mbox{ (}
m
\mbox{ symbols }
\eta
\mbox{).}
$$
It is easy to observe that $\eta_i^2 (\ve_i) = \ve_i$,
$\eta_i^2 (\vf_i) = \vf_i$
and $\eta_i^2 (\ve_j) = \pm \ve_j$, 
$\eta_i^2 (\vf_j) = \pm \vf_j$
with the same sign for $i\neq j$. Let us define
the signs $\epsilon_{i} \in \{-1,1\}$ by
$$
\eta_1^2 (\ve_2) = \epsilon_1 \ve_2, \ \ \
\eta_2^2 (\ve_1) = \epsilon_2 \ve_1\, .
$$
A real root $\alpha$ can be written as
$\alpha = w(\alpha_j)$ for unique $j$ and $w=[w_sw_t]_m\in W$.
We define the signs by
$$
\epsilon_{i,\alpha} \coloneqq 
\begin{cases}
 1,\ \text{if}\ i \neq s \, ,\\
 \epsilon_i^{k_t},\ \text{if}\ i = s, \ w_i (\alpha) = k_1 \alpha_1 + k_2 \alpha_2 \, .
\end{cases}
$$
Notice that this agrees with the definition
in Carter and Chen \cite{CaCh}
because we can define the 
corresponding root element in $\mg$
by $\ve_\alpha \coloneqq [\eta_s \eta_t]_m (\ve_j)$
so that
$$
\eta_i (\ve_\alpha) =
\eta_i [\eta_s \eta_t]_m (\ve_j) =
\begin{cases}
 [\eta_t \eta_s]_{m+1} (\ve_j) = \ve_{w_i (\alpha)},\ \text{if}\ i \neq s \, , \\
 \eta_i^2 [\eta_t \eta_s]_{m-1} (\ve_j) =
 \eta_i^2 (\ve_{w_i (\alpha)}) =
 \epsilon_i^{k_t} \ve_{w_i (\alpha)}, \ \text{if}\ i = s \, .
\end{cases}
$$

Notice that the relation~(5) is simpler in our case
(rank 2 and $\max (a_{21},a_{12}) \leq -2$)
than in the general case \cite{CaCh} 
for two reasons.
First, the conditions in the relation~(5) describe precisely
all prenilpotent pairs of roots $\alpha,\; \beta$.
Second, $\alpha+\beta$ is no longer a root
so that all our commutator relations are trivial.

Let $\F=\F_q$. Note that $U_\a \cong (\F_q,+)$. 
We also have subgroups
$$U_{-}=\langle U_\a \ |\ \a \in \Phi_{-}\rangle
\ \text{ and }\
U:=U_+=\langle U_\a \ |\ \a \in \Phi_+ \rangle .$$ 
The group $G$ admits a $(B,N)$-pair structure, with
$$
B=U \rtimes H, \ \ 
H = \sY \otimes_{\Z} \F^{\times} \ \mbox{and} \ 
N \geq H\ \mbox{with}\ N/H\cong W \ \ 
$$
where $H=B \cap N$.
Then the standard maximal parabolic subgroups are
$$P_i \coloneqq B \amalg B \dot{w_i} B\ \mbox{for}\ i=1,2.$$

The $(B,N)$-pair structure yields an associated Tits building $\sX$ of $G$.
In our case (a Kac-Moody group of rank 2 over a finite field $\F_q$)
the building $\sX$ is a $(q+1)$-regular tree.
Let $\sX_0$ and $\sX_1$ denote respectively the set of vertices and the set of edges of $\sX$. The elements of $\sX_0$ correspond to the conjugates of the parabolics $P_i$, $i=1,2$, and elements of $\sX_1$ to the conjugates of $B$.
In particular, we have two types of vertices - corresponding respectively to $P_1$ and $P_2$. The group $G$ acts naturally on $\sX$ with fundamental domain an edge. The action on $\sX$ yields an action topology,  called
{\em the building topology} on $G$ and
the corresponding completion $\widecheck{G}$.
The kernel of the natural map $G\rightarrow\widecheck{G}$
is non-trivial: it is equal to the centre of $G$.

Caprace and R\'{e}my \cite{CaRe} define a local version of the same topology.
The completion $\overline{G}$ of $G$ with respect to this topology
retains the centre.
Let $\vc \in \sX_1$ be the edge fixed by $B$. For each $n \in \N$ define
$$U_{+,n} \coloneqq \{ g \in U \ |\ g \cdot \vc'=\vc' \ \text{for every edge}\ \vc'\ \text{such that}\ d(\vc,\vc') \leq n\},$$
where $d$ denotes the distance on $\sX$. Using this Caprace and R\'emy define the following left-invariant metric $d_+: G \times  G \to \R_+$ on $G$:
$$
d_+(g,h)=\\
\begin{cases}
 2,\ \text{if}\ g^{-1} h \notin U\ ,\\
 2^{-n},\ \text{if}\ g^{-1} h \in U\ \text{ and }\ n=\text{max}\{k \in \N\ |\ g^{-1}h \in U_{+,k}\}
\end{cases}$$
for all $g, h \in G$ \cite{CaRe}.
Write $\overline{G}$ for the completion of $G$ with respect to this metric.
It is worth noticing that $\overline{G}$ is a locally compact totally disconnected group with a $(B,N)$-pair
$(\overline{B}, N)$, where $\overline{B}$ denotes the closure of $B$ in $\overline{G}$ \cite{CaRe}.

Capdeboscq and Rumynin \cite{CaRu} introduce the local version
of the pro-$p$-topology on $G$. 
Let
$$ \sF \coloneqq \{ A \leq U\ |\ |U:A|=p^k,\ \text{for some}\ k \in \N \}.$$
The set $\sF$ is a fundamental system of neighbourhoods of $1$ in $U$ and, thus, gives a  topology on  $B$.
This also defines a topology on $G$ \cite[Th. 1.2]{CaRu}.
Capdeboscq and Rumynin show that the completion $\widehat{G}$ of $G$ with respect to this topology is a locally compact totally disconnected 
group with $(B,N)$-pair $(\widehat{B},N)$, where $\widehat{B}$ is the completion of $B$ \cite{CaRu}.   Moreover, $\widehat{B}$ is open in $\widehat{G}$ and $\widehat{B}= \widehat{U} \rtimes H$, with $\widehat{U}$ being the full pro-$p$ completion of $U$ \cite{CaRu}.
We will call the group $\widehat{G}$
\emph{the local pro-$p$ completion} of $G$.

\section{Behaviour of elements of order $p$ in $\widehat{G}$}\label{two}

Having introduced the groups $\widehat{G}$ and $\overline{G}$, 
we are ready to investigate their cocompact lattices.
To discuss these two groups in parallel,
we talk about a group $\widetilde{G}\in\{\widehat{G},\overline{G}\}$.  

Recall that \emph{a cocompact lattice} is a discrete subgroup $\G \leq \widetilde{G}$ such that the quotient topological space $\G \setminus \widetilde{G}$ is compact. The space $\G \setminus \widetilde{G}$ admits a finite $\widetilde{G}$-invariant measure (cf. \cite[Ch. 1]{BaLu}).

%
Using  $\Phi_+^1$ and $\Phi_+^2$ from Section~\ref{one}, let us define the following abelian $p$-groups:
$$
U_i \coloneqq \langle U_\alpha \ |\ \alpha \in \Phi_+^i \rangle \ 
\mbox{ and} \ 
-U_i \coloneqq \langle U_\alpha \ |\ -\alpha \in \Phi_+^i \rangle ,\ \text{for}\ i=1,2.$$
We may now consider
$$\sU =   \sU_1 \coloneqq \cl(U_1 \times  -U_2)\ \text{ and }\ \sU_2 \coloneqq \cl( -U_1 \times  U_2), $$
where $\cl( \quad)$ denotes the closure in the relevant topology on the complete Kac-Moody group.
We will be interested in the conjugates of $\sU$.
In particular, note that $\sU_2= w_1 \sU w_1^{-1}$.



\begin{defn} \label{properties}
  We say that a complete Kac-Moody group $\widetilde{G}$ is
  \emph{$p$-well-behaved}, if the following conditions hold:
\begin{itemize}
\item[(P1)] Cocompact lattices in $\widetilde{G}$ do not contain elements of order $p$. 
\item[(P2)] Any element of order $p$ in $\widetilde{G}$ is contained
  in a conjugate of the subgroup $\sU$.
\end{itemize} 
\end{defn}



The aim of this section is to show that the local pro-$p$ complete group $\widehat{G}$ is $p$-well-behaved. It is an open conjecture that
$\overline{G}$ is $p$-well-behaved \cite{CaTh}.


\begin{lemma} \label{free pr decomp}
  Let $U=\langle U_\a \ |\ \a \in \Phi_+ \rangle $, and  $U_1$ and $U_2$  are as above.
  Then $U$ is a free product of $U_1$ and $U_2$.
\end{lemma}
\begin{proof} By \cite[Prop. 4]{Ti2},
$U$ is an amalgamated product of $U_1$ and $U_2$ along the intersection $U_0=U_1 \cap U_2$. However, 
$U_0=1$ \cite{CaTh}. The result follows.
\end{proof}

Now we are ready to tackle Property~(P2):
\begin{thm} \label{P2_hatG}
  Any element of order $p$ in $\widehat{G}$ is contained in a conjugate
  of the subgroup ${\sU}$ of $\widehat{G}$.
\end{thm}
\begin{proof}
  Let $g \in \widehat{G}$ be an element of order $p$. Then $g$ lies in a conjugate of the Sylow pro-$p$-subgroup $\widehat{U}$ of $\widehat{G}$. Thus,
  without loss of generality we may assume that $g \in \widehat{U}$.
By Lemma \ref{free pr decomp}, $U=U_1 \ast U_2$ and, thus, $\widehat{U}=\widehat{U_1 \ast U_2}$.
Let $\amalg$ denote the free pro-$p$ product \cite[9.1]{RiZal}.
Notice that the pro-$p$ completion commutes with the free product \cite[9.1.1]{RiZal}:
$$ \widehat{U}=\widehat{U_1 \ast U_2} \cong \widehat{U_1} \amalg \widehat{U_2}.$$
Hertfort and Ribes \cite{HeRi} show that if a group decomposes as a free pro-$p$ product of two groups, then all the torsion is contained in a conjugate of one of the factors. Hence, $g$ is contained in a conjugate of one of $\widehat{U_i}$, $i=1,2$. Since ${\sU}= \cl(U_1 \times  -U_2)$, the proof is now complete. 
\end{proof}

Our next step is to explain why Property (P2) implies Property (P1). We begin our investigation with a lemma.

\begin{lemma}\label{isomorphic}
Let $G$ be a minimal Kac-Moody group of rank $2$ over $\F_q$, $\widehat{G}$ its local pro-$p$ completion and $\overline{G}$ its Caprace-R\'{e}my completion. 
Let $$
C \coloneqq C(\widehat{G}, \widehat{U}) = \bigcap_{ g \in \widehat{G}} g \widehat{U} g^{-1} \; . $$
Then 
$$\widehat{G}/ C \cong \overline{G}$$
as topological groups.
\end{lemma}
\begin{proof}
Recall that for two groups to be isomorphic as topological groups, we need to show existence of a continuous abstract group isomorphism with a continuous inverse. 
Let 
$\pi: \widehat{G} \twoheadrightarrow \overline{G}$ be the natural map. This is an open continuous homomorphism with kernel $\ker(\pi)=C$ \cite{CaRu}. 
Consider the group $\widehat{G}/ C$ as a topological group with respect to the quotient topology coming from $\widehat{G}$. 
We have a commutative diagram: 
$$\begin{tikzcd}
  \widehat{G} \arrow[r, "\pi"] \arrow[dr, "\theta"]
  & \overline{G}\\
  &\widehat{G}/ C \arrow[u, "\widebar{\pi}"]
 \end{tikzcd}$$
where $\theta$ is the quotient map and $\widebar{\pi}$ is the natural map induced by $\pi$ and factoring through $\theta$. Note that since $\ker(\pi)=C$, $\widebar{\pi}$ is in fact an isomorphism of abstract groups. Moreover, the map $\theta$ is a continuous, open surjection \cite[5.16, 5.17]{HR}. Thus, if $N \subseteq \overline{G}$ is open, 
$$\widebar{\pi}^{-1}(N)=\theta \big(\pi^{-1}(N) \big)$$
is open, giving the continuity of $\widebar{\pi}$.
Furthermore, $\widebar{\pi}$ also has a continuous inverse. Suppose $L \subseteq \widehat{G}/ C$ is open and $\psi \coloneqq \widebar{\pi}^{-1}$, then
$$ \psi^{-1}(L)=\pi \big(\theta^{-1}(L) \big), $$
which is open by continuity of $\theta$, openness of $\pi$ and commutativity of the diagram.
\end{proof}

Since $C \leq \widehat{U}\leq  \widehat{G}$, where $\widehat{U}$ is a compact open and closed (since $\widehat{G}$ is Hausdorff)
subgroup of $G$, 
$C$ is a closed compact pro-$p$ subgroup of $\widehat{G}$.

Observe also that, as $C$ is compact, the quotient map
$\theta: \widehat{G}\rightarrow \widehat{G}/C$ is closed \cite[5.18]{HR}.
It follows from Lemma~\ref{isomorphic} that
the natural map
$\pi: \widehat{G}\rightarrow \overline{G}$ is closed.


\begin{lemma} \label{first_count}
The group $\widehat{G}$ is first countable.
\end{lemma}
\begin{proof}
  It suffices to show that $\widehat{U}$ is first countable
  since $\widehat{U}$ is open in $\widehat{G}$.

  To establish the first countability of $\widehat{U}$, it is sufficient to show
  that it admits a countably infinite generating set which converges to 1 \cite[Rem. 2.6.7]{RiZal}. By Lemma \ref{free pr decomp}, 
  $$ \widehat{U}=\widehat{U_1 \ast U_2} \; .$$
  The group $U_i$ is a countably dimensional vector space over the field $\F_q$.
  Its $\F_p$-basis $e^{(i)}_k$ forms a countable generating sequence, converging
  to 1 in the pro-$p$-topology.
  It follows that  $e^{(1)}_1 , e^{(2)}_1,  e^{(1)}_2 , e^{(2)}_2 \ldots$ is 
  a countable generating sequence of $\widehat{U}$, converging to 1. 
%
\end{proof}

Since $\widehat{G}$ is first countable and, thus, metrizable,
its topology is determined by sequences.
In particular, in a metrizable topological space
compactness is equivalent to sequential compactness. 

\begin{prop}(cf. \cite[Cor. 4.3.]{CaTh}) \label{limit point}
  Let $u$ be an element of  ${\sU}$.
  Then there exists $g \in \widehat{G}$, such that the sequence 
$$x_n \coloneqq g ^n u g^{-n}\, , \ n \in \N $$
has a limit point in $C$.
\end{prop}
\begin{proof}
%
%
%

Since $\pi$ is surjective,  \cite[Cor. 4.3]{CaTh} implies that there exists $g \in \widehat{G}$, such that
$$\lim_{n \to \infty}\pi(x_n)=\lim_{n \to \infty} \pi(g)^{n} \pi(u) \pi(g)^{-n} = 1_{\overline{G}}. $$

  Take a compact open subgroup $K\leq \widehat{G}$.
  Since $\pi (K)$ is a neighbourhood of $1_{\overline{G}}$,
  there exists $N\in\mathbb{N}$, such that $\pi(x_n) \in \pi(K)$
  for all $n>N$.
  Hence, $x_n \in K C$
  for all $n>N$.

  It remains to observe that ${K}  C$
  is compact, and  so $(x_n)$ contains a convergent subsequence $(y_n)$
  in $K C$. Its limit $z=\lim y_n$ must belong to
  $C$ because
   $\pi(z) = \lim \pi(y_n) = \lim \pi(x_n) = 1_{\overline{G}}$.

\end{proof}

We are nearly ready to prove property (P1) for $\widehat{G}$.
The final ingredient we need is the following result for cocompact lattices:

\begin{lemma} \label{CoLa} (cf. \cite[p. 10]{GGPS})
  Let $\G \leq \widehat{G}$ be a cocompact lattice.
For each $u \in \G$ its conjugacy class
$$u^{\widehat{G}} \; =\;  \{g^{-1} u g\; |\; g \in \widehat{G} \}$$
is a closed subset of $\widehat{G}$. 
\end{lemma}
\begin{proof}
  Let us show that $\G$ admits a compact fundamental domain $\widetilde{K}$ in $\widehat{G}$,  i.e.,
  a compact subset $\widetilde{K}$ such that $\widehat{G}= \G \widetilde{K}$.
  Take a compact open subgroup $K \leq \widehat{G}$.
  Consider the quotient map 
  $\theta: \widehat{G} \twoheadrightarrow \G \setminus \widehat{G}$.
  For each $x\in\widehat{G}$ the set
  $\G x K$ is open and $\G$-equivariant.
  By the definition of the quotient topology,
  $\theta (x) K = \theta (\G x K)$ is open in $\G \setminus \widehat{G}$.
  Thus, $\{ \theta (x) K \,\mid\, x\in\widehat{G} \}$
  is an open cover of $\G \setminus \widehat{G}$.
  But $\G$ is cocompact, so we can choose a finite subcover
$\{ \theta (x_i) K \,\mid\, i = 1, \ldots, n \}$.
  It follows that $\widetilde{K}\coloneqq \bigcup_{i=1}^n x_iK$ is a compact fundamental domain.

  The rest of the argument follows Gelfand, Graev and Piatetsky-Shapiro \cite[p. 10]{GGPS}.
  Take $x\in \cl(u^{\widehat{G}})$.  Since $\widehat{G}$ is first countable, there exists a sequence
   $(g_i^{-1} u g_i)$ with $g_i \in \widehat{G}$, $i\in\mathbb{N}$,
  convergent to $x$.
  Let us write each $g_i$ as  $u_i k_i$ for some $u_i\in \G$, $k_i\in \widetilde{K}$.
  Since $\widetilde{K}$ is compact and first countable,
  we can choose a convergent subsequence of $(k_i)$. Thus, without loss
  of generality $(k_i)$ converges to some $k\in\widetilde{K}$. Observe that
  $$
  u_i^{-1} u u_i =
  k_ig_i^{-1} u g_ik_i^{-1} =
  k_i(g_i^{-1} u g_i)k_i^{-1} \longrightarrow
    k(\lim g_i^{-1} u g_i)k^{-1} = 
  kxk^{-1}\; .
  $$
  Since $\G$ is discrete, $u_n^{-1} u u_n = kxk^{-1}$
  for all sufficiently large $n$, so that
  $x\in u^{\widehat{G}}$.
\end{proof}

Note that since the map $\pi: \widehat{G} \to \overline{G}$ is closed,
the set $\pi(u^{\widehat{G}})=\pi(u)^{\overline{G}}$ is closed 
for every $u$ from any cocompact lattice  $\G$. 

\begin{thm} \label{P1_hatG}
  A cocompact lattice in $\widehat{G}$ does not contain elements of order $p$.
\end{thm}
\begin{proof}
Let $\Gamma$ be a cocompact lattice.
  Consider $u \in \G$ such that $|u|=p$.
  By the proof of Theorem~\ref{P2_hatG}, $u$ is contained in a conjugate of
  $\widehat{U_1}$ or $\widehat{U_2}$.
  Without loss of generality, $u\in\widehat{U_1}$.

  By Proposition~\ref{limit point} there exists a $g \in \widehat{G}$, such that the sequence
$$x_n \coloneqq g^n u g^{-n},\ \ n \in \N ,$$
    has a limit point $x\in C$.
    By construction, the sequence $(x_n)$ lies in the closed set $u^{\widehat{G}}$.
    Thus $x \in u^{\widehat{G}}\cap C$.
    Since $C$ acts trivially on the Tits building of $G$, so does $x$.
    This is a contradiction: non-identity elements of  $\widehat{U_1}$
    do not act trivially on the Tits building of $G$.
\end{proof}

Theorem \ref{P1_hatG} and Theorem \ref{P2_hatG} show that $\widehat{G}$ is $p$-well-behaved.

\section{Push and pull for cocompact lattices} \label{three}

Given a continuous homomorphism $\theta : H \rightarrow J$ of locally compact groups
with a non-discrete kernel,
one should not expect a relationship between lattices in $H$ and $J$.
If $\Gamma$ is a lattice in $H$, then $\theta (\G)$ is not necessarily a lattice. In the opposite direction, if $\G$ is a lattice in $J$, then 
 $\theta^{-1} (\G)$ is never discrete.
 However, we can push and pull cocompact lattices along the quotient map.
  
\begin{prop} \label{hat(G) to bar(G)}
  Let $\theta: H \rightarrow J$ be
  a continuous  homomorphism of locally compact groups with
  a compact kernel $K$ and a closed cocompact image.
  If $\G$ is a cocompact lattice in $H$, then
  $\theta(\G)$ is a cocompact lattice in $J$.
\end{prop}
\begin{proof}
  Let us show  that   $\theta (\G)\setminus J$ is compact. The multiplication defines a continuous bijective map from the fibre product:  
\begin{equation}\label{argument}
\eta: F\coloneqq (\theta (\G)\setminus \theta (H)) \times_{\theta (H)} J\longrightarrow \theta (\G)\setminus J \, , \ \
(\theta (\G) a,b) \mapsto \theta (\G) ab \, .
\end{equation}
It suffices to prove that the fibre product $F$ is compact.
For this we need to find a convergent subnet of an arbitrary net $(\theta (\G)x_i, y_i)_{i\in \bI}$ in $F$.

Since $\theta (H)\setminus J$ is compact, the net $(\theta (H)y_i)_{i\in \bI}$ has a convergent subnet. Without loss of generality, $(\theta (H)y_i)_{i\in \bI}$ itself is convergent. This means that each $y_i$ can be written as $y_i=h_i z_i$ with $h_i\in \theta (H)$ and the net $(z_i)_{i\in \bI}$ convergent in $J$.
Since $\theta (\G)\setminus \theta (H)$ is compact, the net $(\theta (\G)x_ih_i)_{i\in \bI}$ has a convergent subnet $(\theta (\G)x_ih_i)_{i\in \bJ}$.
Finally, since $(\theta (\G)x_i, y_i) =(\theta (\G)x_ih_i, z_i)$, 
it follows  that $(\theta (\G)x_i, y_i)_{i\in \bI}$ is a convergent subnet we sought.

  
It remains to show that $\theta(\G)$ is discrete.
  Suppose not.  
  Then we can pick a net $(x_i)_{i\in \bI}$ in $\G$, convergent to some $a \in \theta(\G)$, such that $\theta(x_i)\neq a$ for all $i$. 
  
Choose a compact neighbourhood of identity $V\subseteq J$.
The inverse image $\theta^{-1}(Va)$ is compact by
the argument for~\eqref{argument} with  the new key map 
$$
\eta: K \times_{K} \theta^{-1}(Va) \longrightarrow Va\cap\theta (H) \, , \ \
(k,c) \mapsto \theta (c) \, .
$$
There exists an ordinal $\bL<\bI$ such that $\theta(x_i) \in  Va$ and, consequently, $x_i \in  \theta^{-1} (Va)$ for all $i \geq \bL$.
Hence, we can find a convergent subnet $(y_j)_{j\in \bJ}$ (of $x_i$).
Since $\G$ is discrete, there exists an ordinal $\bM<\bJ$, such that $y_j = y_{\bM}$ for all $j\geq \bM$. 
Then
$a = \lim \theta (x_i) = \lim \theta (y_j) = \theta (y_{\bM})$, 
a contradiction. 
\end{proof}
We can control the push-forward of cocompact lattices
along the map
$\pi: \widehat{G} \to \overline{G}$
more tightly than for a general map:
\begin{lemma} \label{trivial intersection}
  If $\G$ is a cocompact lattice in $\widehat{G}$,
  then $\G \cap C =\{1_{\widehat{G}}\}$. 
\end{lemma}
\begin{proof}
  First note that $\G \cap C$ is finite since $\G$ is discrete and
  $C$ is compact.
  Since $C$ is the intersection of Sylow pro-$p$-subgroups of $\widehat{G}$,
  every finite order element in $\G \cap C$ must have order $p^k$.
  But $\widehat{G}$ is $p$-well-behaved, in particular,
  cocompact lattices do not contain elements of order $p$.
  It follows that $\G \cap C=\{ 1_{\widehat{G}}\}$.
\end{proof}
\begin{cor}
  \label{cor33}
  If $\G\leq \widehat{G}$ is a cocompact lattice, then
  $\pi:\G \rightarrow \pi (\G)$ is an isomorphism. In particular, the cocompact lattice
  $\pi(\G)\leq \overline{G}$ contains no elements of order $p$.  
  \end{cor}
All the cocompact lattices constructed by Capdeboscq
and Thomas in $\overline{G}$ are subgroups of $G$ \cite[Th. 1.1, Th. 1.2]{CaTh}.
In fact, their main result \cite[Th. 1.3]{CaTh}
can be interpreted as a classification of
edge-transitive (see Section~\ref{five}) cocompact lattices
with no elements of order $p$ in $\overline{G}$.
All of them are conjugate to subgroups of $G$ and can be lifted
using the next proposition:
\begin{prop} \label{bar(G) to hat(G)}
  Let $\G \leq G$ be a cocompact lattice in $\overline{G}$.
  Then $\G$ is also a cocompact lattice in $\widehat{G}$.
\end{prop}
\begin{proof}
  Suppose $\G \leq \widehat{G}$ is not discrete.
  Then there exists a sequence $x_n \in \G$, convergent to $a\in\G$, with $x_n\neq a$ for all $n$.
  Since $\pi (a) = \lim \pi (x_n)$ and $\G$ is discrete in $\overline{G}$,
the sequence $\pi (x_n)$ is eventually constant:
there exists $N$, such that $\pi(x_n)=\pi (a)$ (equivalently, $x_n \in aC$) for all $n\geq N$.
Hence, $x_n \in aC\cap G$ for all $n\geq N$ since $\G\subseteq G$.
  
Now observe that the set $aC\cap G$ has at most one element.
Consider two of its elements $ag$ and $ah$ with $g,h\in C$.
Then $(ag)^{-1}ah = g^{-1}h \in C\cap G$ that is equal to $\{1\}$
because $\G$ is a subgroup of $\overline{G}$ as well.
Inevitably $g=h$ and $|aC\cap G| \leq 1$.
It follows that $x_n =x_N$ for all $n\geq N$, a contradiction with all $x_n\neq a$. This shows that $\G$ is discrete in $\widehat{G}$.

The space $\G\setminus\widehat{G}$ is compact because both $C$ and $\G\setminus\overline{G}\cong \G\setminus\widehat{G}/C$ are compact.
The proof is identical to the argument for~\eqref{argument} with  the key map defined by 
$$
\eta: \G\setminus\widehat{G} \times_{C} C \longrightarrow \G\setminus\widehat{G} \, , \ \
(\G a,c) \mapsto \G ac \, .
$$
\end{proof}

\section{Covolumes} \label{four}
Following Bass and Lubotzky \cite{BaLu}, 
let us discuss how the processes of pushing and pulling
cocompact lattices affect their covolumes. 
Recall that for a locally compact group $H$ acting on a set $X$
with compact open stabilisers $H_\vx$, $\vx \in X$ there is a natural \emph{covolume} of a discrete subgroup $\G\leq H$ defined by
$$\voll (\G \setminus \setminus X) \coloneqq \sum_{[\vx] \in \G \setminus X} \frac{1}{|\G_\vx|},$$
where $\G_\vx=H_\vx \cap \G$.
Moreover,
$\voll (\G \setminus \setminus X)< \infty$ if and only if
$\Gamma$ is a lattice (forcing $H$ to be unimodular) and
$$
\mu (H\setminus\setminus X) \coloneqq \sum_{[\vx] \in H \setminus X} \frac{1}{\mu(H_\vx)} < \infty,$$
where $\mu$ is
a right-invariant Haar measure on $H$. 
In this case, we can choose 
$\mu$ on $H$ in such a way that (see \cite[1.5]{BaLu})  
$$\voll (\G \setminus \setminus X) = \mu_{\G\setminus H} (\G \setminus H).$$


We wish to compare covolumes of cocompact lattices in $\widehat{G}$ and $\overline{G}$.
As described in Section~\ref{one}, the Tits building $\sX$ of a rank 2 Kac-Moody group is a tree, whose 
set of vertices $\sX_0$ consists of conjugates of the parabolic subgroups $P_1$ and $P_2$. Let $\vx_i$ denote the vertex corresponding to $P_i$ and $[\vx_i]$ its equivalence class under the action of $G$, $i=1,2$. Then
$$ G \setminus \sX_0 = [\vx_1] \sqcup [\vx_2].$$ 
Both $\overline{G}$ and $\widehat{G}$ act on $\sX$ (cf. the discussion after \cite[Cor. 1.4]{CaRu}). Abusing notation we also write $[\vx_i]$ for the $\overline{G}$ and $\widehat{G}$ equivalence classes of $\vx_i$.


\begin{prop}
\label{Haar_comparison}  
It is possible to normalise the Haar measures
$\widehat{\mu}$ on $\widehat{G}$ and $\overline{\mu}$ on $\overline{G}$
in such a way that
  $$
  \widehat{\mu} ({\G\setminus\widehat{G}}) =
    \overline{\mu} ({\G\setminus\overline{G}}) = 
\sum_{[\vx] \in \G \setminus \sX_0} \frac{1}{|\G_\vx|}
$$
for any cocompact lattice $\Gamma \leq \widehat{G}$,
where, by abuse of notation, 
$\widehat{\mu}$ and $\overline{\mu}$ are also the induced measures on $\Gamma\setminus\widehat{G}$ and $\Gamma\setminus\overline{G}$
correspondingly.
\end{prop}  

\begin{proof}
  If there exist no cocompact lattices, the statement is 
  tautologically true. 


Let $\G \leq \widehat{G}$ be a cocompact lattice. Then $\pi(\G) \cong \G$ is a cocompact lattice in $\overline{G}$.
The group $\overline{G}$ acts on $\sX$ and
the stabilisers $\overline{G}_{\vx}$ for all $\vx \in \sX_0$
are compact open subgroups. In particular,
$\overline{\mu}(\overline{G}_\vx) < \infty$. Thus,
$$\mu (\overline{G}\setminus\setminus \sX) = \sum_{[\vx] \in \overline{G} \setminus \sX_0} \frac{1}{\overline{\mu}(\overline{G}_\vx)} = \frac{1}{\overline{\mu}(\overline{G}_{\vx_1})} + \frac{1}{\overline{\mu}(\overline{G}_{\vx_2})} < \infty,$$
where $\vx_1$ and $\vx_2$ are representatives of the orbits of $P_1$ and $P_2$ under the $\overline{G}$-action respectively.
It follows that $\overline{\mu}$ can be normalised so that 
$$
\overline{\mu} ({\G\setminus\overline{G}}) = 
\sum_{[\vx] \in \pi(\G) \setminus \sX_0} \frac{1}{|\pi(\G)_\vx|}=\sum_{[\vx] \in \G \setminus \sX_0} \frac{1}{|\G_\vx|}.
$$
Now consider the orbits of $\vx_1$ and $\vx_2$ under the action of $\widehat{G}$. Again, we have compact open stabilisers $\widehat{G}_\vx$, for every $\vx \in \sX_0$. By the same argument as above
$$\mu (\widehat{G}\setminus\setminus \sX) = \sum_{[\vx] \in \widehat{G} \setminus \sX_0} \frac{1}{\widehat{\mu}(\widehat{G}_\vx)} < \infty.$$
Consequently,
$$\widehat{\mu} ({\G\setminus\widehat{G}}) = 
\sum_{[\vx] \in \G \setminus \sX_0} \frac{1}{|\G_\vx|}=\overline{\mu} ({\G\setminus\overline{G}}) 
$$
as required.
\end{proof}

\section{Cocompact lattices in $\widehat{G}$, symmetric case}
\label{five}
In this section we assume that $A$ is symmetric ($a_{12}=a_{21}$).
There is a unique (up to an isomorphism) simply connected root datum $\sD_{sc}$ of type $A$:
this is a root datum such that $\Pi^\vee$ forms a basis of $\sY$.
Let $G$ be the Kac-Moody group with the simply connected root datum.
It is possible, yet requiring extra work beyond the scope
of the present paper, 
to extend our results
to an arbitrary root datum $\sD$. 

A lattice $\G$ is called {\em edge-transitive}, if it acts transitively on the set $\sX_1$ of edges of the Tits building, described in Section~\ref{one}. 
Capdeboscq and Thomas classify edge-transitive
$p$-well-behaved cocompact lattices in
$\overline{G}$ \cite{CaTh}.
Now Corollary~\ref{cor33} and \cite[Th. 1.3]{CaTh}
together give us a classification
of edge-transitive
cocompact lattices in $\widehat{G}$.

Capdeboscq and Thomas also determine the $p$-well-behaved cocompact lattice
of the minimal covolume in $\overline{G}$
\cite[Th. 1.4]{CaTh}. This and the observations above yield 
Main Theorem~\ref{main_th},
a similar result for all
  cocompact lattices in $\widehat{G}$.

  Let us prove Main Theorem~\ref{main_th} now.
  Statement (1) follows from  
  Corollary~\ref{cor33} 
  and \cite[Th. 1.1]{CaTh}.
Statement (2) follows from
Proposition~\ref{Haar_comparison} and \cite[Th. 1.3]{CaTh}.
Q.E.D.

\section{Cocompact lattices in $\widehat{G}$, other cases}
\label{six}

\subsection{Not symmetric}
Let us drop the assumption that $a_{21}=a_{12}$
while still assuming that $\max (a_{21},a_{12}) \leq -2$.
These are our assumptions from Section~\ref{one} through Section~\ref{four}.
Thus, all of our results in these sections hold.

Looking at Section~\ref{five},
statement~(1) of Main Theorem~\ref{main_th}
holds but validity of statement~(2) is unclear
at this time.
The reason is that 
Capdeboscq and Thomas \cite{CaTh} do most of their analysis
only in the symmetric case.
While the first two statements of 
\cite[Th. 1.1]{CaTh}
hold without the symmetricity assumption,
yet \cite[Th. 1.3]{CaTh}
requires it. With this in light
we find the following question interesting.
\begin{que}
  Determine a cocompact lattice of minimal covolume in $\widehat{G}$
  and compute its covolume.
\end{que}

\subsection{The case of $a_{12}=-1$}
We still assume that
$\max (a_{21},a_{12}) \leq -2$.
The first issue with this case is that our definition
of Kac-Moody group does not work in this case.
If $m=a_{21}$, the reflection representation of the Weyl group
 in this case is given by
$$ \rho (w_1)= 
\begin{pmatrix}
-1 & m \\
0 & 1 \\
\end{pmatrix}\, , \ \ 
\rho (w_2) =
\begin{pmatrix}
 1 &  0 \\
 1 & -1
\end{pmatrix} \, .$$
Hence, we have a real root that is a sum of two real roots:
$$
w_1 (\alpha_2) =
\begin{pmatrix}
 m \\
 1
\end{pmatrix}
=
\begin{pmatrix}
 1 \\
 0
\end{pmatrix}
+
\begin{pmatrix}
 m-1 \\
 1
\end{pmatrix}
=
\alpha_1 + w_1w_2(\alpha_1).
$$
Consequently, there exist non-commuting subgroups $U_\alpha$
and $U_\beta$ for prenilpotent pairs $\{\alpha, \beta \}$.
Nevertheless, 
the groups $\widehat{G}$ and
$\overline{G}$ still exist in this case \cite{CaRu}
and 
many of our general results about them, e.g.,
Lemmas~\ref{isomorphic} and \ref{CoLa}, 
Propositions~\ref{limit point}, 
\ref{hat(G) to bar(G)}, 
\ref{bar(G) to hat(G)} and  
\ref{Haar_comparison}
are still applicable.
The group $\widehat{G}$ is still first countable,
although Lemma~\ref{first_count}
requires a new, subtler proof.

The current proof does not work because 
Lemma~\ref{free pr decomp}, a key result for all consequent analysis,
fails.  
Hence, we do not know whether many major results, e.g., 
Theorems~\ref{P2_hatG} and \ref{P1_hatG}, 
Lemma~\ref{trivial intersection}, 
Corollary~\ref{cor33}
remain valid.
We state some of them in a series of conjectures and questions.
  

\begin{con}  The Kac-Moody group $\widehat{G}$ admits a cocompact lattice.
\end{con}

\begin{con}  A cocompact lattice in $\widehat{G}$ does not contain an element of order $p$. 
\end{con}

\begin{que} Classify cocompact lattices in $\widehat{G}$. \end{que}

\begin{que}
  Determine a cocompact lattice of minimal covolume in $\widehat{G}$
  and compute its covolume.
\end{que}


\subsection{Higher rank}
Let us now consider a generalised Cartan matrix $A$ of a larger size.
The groups $\widehat{G}$ and
$\overline{G}$ exist in this case \cite{CaRu}.
We expect that the group $\widehat{G}$ is still first countable
(cf. Lemma~\ref{first_count}).
Consequently, the general results about them, e.g.,
Lemmas~\ref{isomorphic} and \ref{CoLa}, 
Propositions~\ref{limit point}, 
\ref{hat(G) to bar(G)}, 
\ref{bar(G) to hat(G)}
and 
\ref{Haar_comparison}
are still applicable.

It would be interesting to know when these groups admit cocompact lattices.
If $A$ is affine, not of type $\widetilde{A_n}$,
Borel and Harder prove 
that $\overline{G}$ does not admit cocompact lattices \cite{BH}. 
Now suppose that $A$ has an irreducible principal minor that is
affine, not of type $\widetilde{A_n}$.
Caprace and Monod observe
that 
$\overline{G}$
does not admit  
a cocompact lattice either \cite[Rem. 4.4]{CMo}. 
Proposition~\ref{hat(G) to bar(G)} implies
that 
$\widehat{G}$
does not admit  
a cocompact lattice either.
On the other hand, cocompact lattices exist in
a right-angled Kac-Moody group $\overline{G}$
\cite{CaTh2}. The right-angled Kac-Moody
groups are a subclass of the groups in the following conjecture:

\begin{con} If any irreducible principal minor of $A$ of affine type
is of type $\widetilde{A}_m$,
then
the Kac-Moody group $\widehat{G}$ admits a cocompact lattice.
\end{con}

\begin{que}
  Determine a cocompact lattice of minimal covolume in $\widehat{G}$
  and compute its covolume.
\end{que}


\end{document}